\documentclass[a4paper,12pt]{article}

\usepackage{color}

\usepackage{mymacros}

\newcommand{\PO}{\mathop{PO}\nolimits}

\newcommand{\NE}{\operatorname{NE}}

\newcommand{\ch}{\operatorname{ch}}

\newcommand{\cYv}{{\Check{\cY}}}

\newcommand{\Mir}{\operatorname{Mir}}
\newcommand{\amb}{{\mathrm{amb}}}
\newcommand{\reg}{{\mathrm{reg}}}
\newcommand{\vc}{{\mathrm{vc}}}

\newcommand{\NS}{\operatorname{NS}}

\newcommand{\lambdabar}{\overline{\lambda}}

\newcommand{\Deltaul}{\underline{\Delta}}
\newcommand{\Aff}{\operatorname{Aff}}
\newcommand{\cFtilde}{\widetilde{\cF}}

\newcommand{\bsnablabar}{\overline{\bsnabla}}

\newcommand{\PD}{\operatorname{PD}}

\newcommand{\XFA}{X_{\cF(A)}}
\newcommand{\XFzero}{{X_{\cF(A_0)}}}
\newcommand{\XFone}{{X_{\cF(A_1)}}}
\newcommand{\XFtildeA}{X_{\cFtilde(A)}}

\newcommand{\XtildeFzero}{\Xtilde_{\cF(A_0)}}
\newcommand{\XtildeFone}{\Xtilde_{\cF(A_1)}}

\newcommand{\MSzero}{{\cM_{\Sigma_0}}}
\newcommand{\MSone}{{\cM_{\Sigma_1}}}
\newcommand{\bsNv}{\check{\bsN}}
\newcommand{\bsMv}{\check{\bsM}}

\newcommand{\HAZ}{H_{A,\bZ}}
\newcommand{\HBZ}{H_{B,\bZ}}

\newcommand{\Euler}{\operatorname{Euler}}

\newcommand{\SqMat}[4]{\begin{pmatrix}#1&#2\\#3&#4\end{pmatrix}}
\newcommand{\sqmat}[4]{\begin{pmatrix}#1&#2\\#3&#4\end{pmatrix}}

\newcommand{\cDv}{\check{\cD}}
\newcommand{\cDM}{\cD}
\newcommand{\GammaM}{\Gamma}
\newcommand{\fYv}{\check{\fY}}

\title{
Modular surfaces
associated with\\
toric K3 hypersurfaces}
\author{Kenji Hashimoto,
Atsuhira Nagano,
Kazushi Ueda}
\date{}
\begin{document}

\maketitle

\begin{abstract}
We give detailed descriptions of the period maps
of two 2-parameter families of anti-canonical hypersurfaces
in toric 3-folds.
One of them is related to a Hilbert modular surface,
and the other is related to the product of modular curves.
\end{abstract}


\section{Introduction}

Mirror symmetry is a mystesious relationship
between complex geometry and symplectic geometry
motivated by string thoery.
Although mirror symmetry is most intensively studied
for Calabi-Yau 3-folds,
mirror symmetry for other classes of varieties,
such as abelian varieties, Fano varieties,
or varieties of general type,
is also an interesting subject.

In this paper, we study the period maps
of two 2-parameter families of K3 surfaces
from the point of view of mirror symmetry.
Similar analysis for the 1-parameter family
of the quartic mirror K3 surfaces has been performed
by Hartmann \cite{MR3062592},
based on earlier results by
Nagura and Sugiyama \cite{Nagura-Sugiyama} and
Narumiya and Shiga \cite{Narumiya-Shiga}.

The first family
is mirror to anti-canonical hypersurfaces
in the $\bP^1$-bundle
$\bP(\cO_{\bP^2} \oplus \cO_{\bP^2}(2))$
over $\bP^2$.
The parameter space of this family
admits a natural compactification $\XFzero$
called the \emph{secondary stack}.
The Picard lattice and the transcendental lattice
of a very general member of this family are given by
\begin{align}
 M_0 &= E_8 \bot E_8 \bot
\begin{pmatrix}
 2 & 1 \\
 1 & -2
\end{pmatrix},
\qquad
 T_0 = U \bot
\begin{pmatrix}
 2 & 1 \\
 1 & -2
\end{pmatrix}.
\end{align}
The moduli space
$
 \cM_0
$
of $M_0$-polarized K3 surfaces
is the symmetric Hilbert modular surface
associated with $\bQ(\sqrt{5})$.

\begin{theorem} \label{th:main0}
The period map gives an isomorphism
\begin{align}
 \Pitilde : \XtildeFzero \simto \MSzero,
\end{align}
of Deligne-Mumford stacks,
where
$
 \XtildeFzero \to \XFzero
$
is a weighted blow-up of weight $(1,3)$ at one point
followed by the root construction along the discriminant, and
$\MSzero$ is a toroidal compactification of $\cM_0$
equipped with a natural orbifold structure.
\end{theorem}

The \emph{root construction} is the operation
introduced in \cite{Abramovich-Graber-Vistoli,Cadman_US}
which gives a generic stabilizer to a divisor.
The fan $\Sigma$ comes from the monodromy logarithm
of the period map
around toric divisors of $\XFzero$.
The inverse image in $\cM_\Sigma$ of the unique cusp
in the Baily-Borel-Satake compactification $\cMbar_0$
is the union of two toric divisors.
The intersection point of these divisors
is the maximally unipotent monodromy point.
The monodromy logarithms $N$ around these divisors
satisfy $N^2 \ne 0$ and $N^3=0$,
so that one has type I\!I\!I degenerations there.

The second family is mirror to anti-canonical hypersurfaces
in the toric weak Fano 3-fold of Picard number 2,
obtained as a crepant resolution
of a toric Fano 3-fold of Picard number 1
with an ordinary double point.
The Picard lattice and the transcendental lattice
of a very general member of this family are given by
\begin{align}
 M_1 &= E_8 \bot E_8 \bot
\begin{pmatrix}
 0 & 3 \\
 3 & 0
\end{pmatrix},
 \qquad
 T_0 = U \bot
\begin{pmatrix}
 0 & 3 \\
 3 & 0
\end{pmatrix}.
\end{align}
The moduli space
$
 \cM_1
$
of $M_1$-polarized K3 surfaces
is the quotient of the product $\bH \times \bH$
of upper half planes
under the action of
$
 \Gamma_0(3) \times \Gamma_0(3) \times C_2.
$
Here each $\Gamma_0(3)$ is a congruence subgroups of $\SL_2(\bZ)$
acting on each upper half plane, and
$C_2$ is the cyclic group of order 2 which permutes
two upper half planes.

\begin{theorem} \label{th:main1}
The period map gives an isomorphism
\begin{align}
 \Pitilde : \XtildeFone \simto \MSone
\end{align}
of Deligne-Mumford stacks,
where
$
 \XtildeFone \to \XFone
$
is the root construction along the discriminant, and
$\MSone$ is a toroidal compactification of $\cM_1$
equipped with a natural orbifold structure.
\end{theorem}

The fan $\Sigma_1$ comes from the monodromy logarithm
of the period map
around toric divisors of $\XFone$.
The cusp in the Baily-Borel-Satake compactification $\cMbar_1$
is the union of two rational curves
intersecting at one point,
and the morphism $\MSone \to \cMbar_1$ is the blow-up
at this intersection point.
There are two maximally unipotent monodromy point in $\MSone$
on the exceptional curve,
corresponding to two crepant resolutions
of the toric Fano 3-fold.
The monodromy logarithm $N$ satisfies
$N \ne 0$ and $N^2=0$ around these divisors,
so that one has type I\!I degenerations there.

The advantage of the secondary stack is that
it comes with a natural family of toric hypersurfaces on it.
This gives a family of degenerate K3 surfaces
on the boundary of the toroidal compactification
of the period domain.
We expect that Theorems \ref{th:main0} and \ref{th:main1}
admits an interpretation
in terms of log period map
of log K3 surfaces
\cite{MR2097359,MR2465224}.

This paper is organized as follows:
In \pref{sc:LPK3},
we recall the notion of lattice-polarized K3 surfaces
and compactifications of their moduli spaces.
In \pref{sc:Dolgachev},
we recall a conjecture of Dolgachev
\cite{Dolgachev_MSK3}
on the relation between polar duality \cite{Batyrev_DPMS}
and mirror symmetry for lattice-polarized K3 surfaces.
In \pref{sc:secondary},
we recall the notion of the secondary stacks
from \cite{MR1976905,
Hacking_CMHA,
Diemer-Katzarkov-Kerr_SGR}.
In \pref{sc:monodromy},
we describe the monodromy of the period map
in terms of the autoequivalence
of the derived category of coherent sheaves
on the mirror manifold
along the lines of \cite{Iritani_QCP}.
%
\pref{th:main0} is proved in \pref{sc:A0}, and
\pref{th:main1} is proved in \pref{sc:A1}.

\emph{Acknowledgment}:
The authors thank Makoto Miura
for collaboration at an early stage of this research;
this paper is originally conceived
as a joint project with him.
A.~N. thanks Hironori Shiga for valuable discussions.
A part of this work is done
while K.~U. is visiting Korea Institute for Advanced Study,
whose hospitality and wonderful working environment
is gratefully acknowledged.
A.~N. is supported by
Waseda University Grant for Special Research Project (2013A-870).
K.~U. is supported by JSPS Grant-in-Aid for Young Scientists No.~24740043.

\section{Lattice polarized K3 surfaces}
 \label{sc:LPK3}

The {\em K3 lattice} is the even unimodular lattice
$
 L = E_8 \bot E_8 \bot U \bot U \bot U
$
of rank 22 and signature $(3,19)$.
Here $E_8$ is the negative-definite even unimodular lattice
of type $E_8$ and
$U$ is the indefinite even unimodular lattice
of rank two.
For a K3 surface $Y$,
set
\begin{align}
 \Delta(Y) = \{ \delta \in \Pic(Y) \mid (\delta, \delta) = -2 \}.
\end{align}
Let $\cL$ be a line bundle
such that $[\cL] = \delta \in \Delta(Y)$.
Riemann-Roch theorem gives
\begin{align}
 h^0(\scL) + h^0(\scL^\vee)
  \ge 2 + \frac{1}{2}(\delta, \delta) =1,
\end{align}
so that $\scL$ or $\scL^\vee$ has a non-trivial section
and hence either $\delta$ or $- \delta$ is effective;
\begin{align}
 \Delta(Y) &= \Delta(Y)^+ \amalg \Delta(Y)^-, \\
 \Delta(Y)^+ &= \{ \delta \in \Delta(Y)
  \mid \delta \text{ is effective} \}, \\
 \Delta(Y)^- &= - \Delta(Y)^+.
\end{align}
The subgroup $W(Y) \subset O(L)$ generated by reflections
with respect to elements in $\Delta(Y)$ acts
properly discontinuously on the connected component
\begin{align}
 V^+ \subset
 V(Y) = \{ x \in H^{1,1}(Y) \cap H^2(Y, \bR)
  \mid (x, x) > 0 \}
\end{align}
containing the K\"{a}hler class.
The fundamental domain is given by
\begin{align}
 C(Y) = \{ x \in V(Y)^+ \mid (x, \delta) \ge 0 \text{ for any }
  \delta \in \Delta(Y)^+ \},
\end{align}
and the K\"{a}hler cone is given
(cf.~e.g.~\cite[Corollary VIII.3.9]{Barth-Hulek-Peters-Van_de_Ven})
by
\begin{align}
 C(Y)^+ = \{ x \in V(Y)^+ \mid (x, \delta) > 0 \text{ for any }
  \delta \in \Delta(Y)^+ \}.
\end{align}
Recall that
\begin{align}
 \Pic(Y) = H^{1,1}(Y) \cap H^2(Y; \bZ)
\end{align}
by the Lefschetz theorem.
Set
\begin{align}
 \Pic(Y)^+ = C(Y) \cap H^2(Y; \bZ), \\
 \Pic(Y)^{++} = C(Y)^+ \cap H^2(Y; \bZ).
\end{align}

Let $M$ be an even non-degenerate lattice
of signature $(1, t)$
where $0 \le t \le 19$.
Choose one of two connected components of
\begin{align}
 V(M) = \lc x \in M_\bR \relmid (x,x) > 0 \rc
\end{align}
and call it $V(M)^+$.
Choose a subset $\Delta(M)^+$ of
\begin{align}
 \Delta(M) = \lc \delta \in M \relmid (\delta, \delta) = -2 \rc
\end{align}
such that
\begin{enumerate}
 \item
$\Delta(M) = \Delta(M)^+ \amalg \Delta(M)^-$
where $\Delta(M)^- = \lc - \delta \relmid \delta \in \Delta(M)^+ \rc$, and
 \item
$\Delta(M)^+$ is closed under addition (but not subtraction).
\end{enumerate}
Define
\begin{align}
 C(M)^+ = \lc h \in V(M)^+ \cap M \relmid
  (h, \delta) > 0 \text{ for all } \delta \in \Delta(M)^+ \rc.
\end{align}

\begin{definition}[{Dolgachev \cite{Dolgachev_MSK3}}]
An {\em $M$-polarized K3 surface} is a pair $(Y, j)$
where $Y$ is a K3 surface and
$
 j : M \hookrightarrow \Pic(Y)
$
is a primitive lattice embedding.
An {\em isomorphism} of $M$-polarized K3 surfaces
$(Y, j)$ and $(Y', j')$ is an isomorphism
$f : Y \to Y'$ of K3 surfaces
such that $j = f^* \circ j'$.
An $M$-polarized K3 surface is
{\em ample} if
\begin{align}
 j(C(M)^+) \cap \Pic(Y)^{++} \ne \emptyset.
\end{align}
\end{definition}


Fix a primitive lattice embedding
$
 i_M : M \hookrightarrow L
$
and let
$
 T
$
be the orthogonal complement.
The period domain
\begin{align} \label{eq:period_domain}
 \cDM = \{ [\Omega] \in \bP(T_\bC)
  \mid (\Omega, \Omega) = 0, \ (\Omega, \Omegabar) > 0 \}
\end{align}
of $M$-polarized K3 surfaces
can be identified with the symmetric homogeneous space
$
 O(2, 19-t) / SO(2) \times O(19-t)
$
of oriented positive-definite 2-planes in $T_\bR$.
It consists of two onnected components
$\cDM^+$ and $\cDM^-$,
each of which is isomorphic to a bounded Hermitian domain
of type IV. Set
\begin{align} \label{eq:GammaM}
 \Gamma(M) = \{ \sigma \in O(L) \mid \sigma(m) = m
  \text{ for any } m \in M \}
\end{align}
and $\GammaM$ be its image
under the natural injective homomorphism
\begin{align}
 \Gamma(M) \hookrightarrow O(T).
\end{align}
Global Torelli Theorem
and the surjectivity of the period map for K3 surfaces show that
the moduli space of ample $M$-polarized K3 surfaces is
bijective with
$
 \cDM^\circ / \GammaM,
$
where
\begin{align}
 \cDM^\circ = \cDM \setminus \lb \bigcup_{\delta \in \Delta(T)}
  H_\delta \cap \cDM \rb
\end{align}
is the complement of reflection hyperplanes
\begin{align}
 H_\delta = \{ z \in T_\bC \mid (z, \delta) = 0 \}.
\end{align}
The closure of the period domain
in the {\em compact dual}
\begin{align}
 \cDv = \{ [\Omega] \in \bP(T_\bC) \mid
  (\Omega, \Omega) = 0 \}
\end{align}
of the period domain
is denoted by
$\cDM^*$.
Its topological boundary is given by
\begin{align}
 \cDM^* \setminus \cDM
  = \bigcup_{I \text{ : isotropic subspace of $M_\bR$}}
     B(I),
\end{align}
where $B(I)$ is defined by
\begin{equation}
 B(I)=\bP(I_{\bC}) \setminus \lb \textstyle{\bigcup_{J \subsetneq I} \bP(J_{\bC})} \rb.
\end{equation}
Since the signature of $M$ is $(2, 19-t)$,
one either has $\rank I = 1$ or 2,
so that $\bP(I_{\bC}) \cap \cDM^*$ is
one point or isomorphic to the upper half plane.
The boundary component is {\em rational}
if $I$ is defined over $\bQ$.
The {\em Satake-Baily-Borel compactification}
is defined by
\begin{align}
 \overline{\cDM / \Gamma}
  = \lb \cDM \cup
       \bigcup_{I \text{ : rational}}
         \bP(I_{\bC}) \cap \cDM^*
      \rb / \Gamma.
\end{align}


Assume that one has
$
 T = U \bot N
$
for a lattice $N$,
and consider the neighborhood of the cusp
corresponding to the isotropic subspace $\bZ e \subset T$.
Let $\{ e, \, f \}$ be a basis of $U$ satisfying
\begin{align}
 (e, e)=(f, f)=0, \ 
 (e,f)=1,
\end{align}
and $\Gamma_e$ be the stabilizer of $e$
in $O(T)$.
With an element $v \in N$,
one can associate an isometry $\varphi_{e,v} \in O(T)$
defined by
\begin{align}
 \varphi_{e,v}(x)
  &= x - \lb \frac{1}{2}(v,v) (e,x)+(v,x) \rb e
   + (e,x) v.
\end{align}
One can easily see that
\begin{align}
 \varphi_{e,v} \circ \varphi_{e,w}
  &= \varphi_{e, v+w},
\end{align}
and
\begin{align}
 \varphi_{e,v}(e) &= e, &
 \varphi_{e,v}(f) &= - \frac{1}{2} (v, v) e + f + v, &
 \varphi_{e,v}(w) &= - (v, w) e + w
\end{align}
for $w \in N$.
It follows that
$\varphi_{e, \bullet}$
gives an embedding
\begin{align} \label{eq:embedding}
 \varphi_{e, \bullet} : N \hookrightarrow O(T)
\end{align}
of groups.
Any element of $\phi \in \Gamma_e$
can be written as
$\psi \circ \varphi_{e,v}$,
where $v \in N$ is defined by
$\phi(f) \equiv f + v$ mod $\bZ e$ and
$\psi \in O(N)$.
This shows that
\begin{align}
 \Gamma_e = O(N) \ltimes N.
\end{align}
The period domain \eqref{eq:period_domain}
can be realized as a tube domain
\begin{align} \label{eq:tube_domain}
 \lc v = v_1 + \sqrt{-1} v_2 \in N_\bC
  \relmid v_i \in N_{\bR},~ (v_2, v_2) > 0 \rc
\end{align}
through the correspondence
\begin{align}
 \Omega = - \frac{1}{2} (v, v) e + f + v.
\end{align}
Under this correspondence,
the action of $\varphi_{e,u} \in N \subset \Gamma_e$ is given by translation
$
 v \mapsto v + u.
$
Hence a neighborhood of the cusp of $\cDM/\Gamma$
is locally isomorphic to
\begin{align}
 \left. \lb N_\bC/N \rb \right/ O(N)^+
  \cong N_{\bCx} / O(N)^+.
\end{align}
Here $O(N)^+$ is the subgroup of $O(N)$ of index $2$
preserving the connected component of $\cDM^+$.
In other words, $O(N)^+$ consists of elements
whose norm and spinor norm have the same sign.
Let $\Sigma$ be a fan in $N$
which is invariant under the action of $O(N)^+$.
Then the toric variety $X_\Sigma$ associated with $\Sigma$
admits a natural action of $O(N)^+$.
By replacing the neighborhood of the cusp
with the neighborhood of the origin in the quotient $X_\Sigma/O(N)^+$,
one obtains a toroidal partial compactification of $\cDM/\Gamma$.

\section{Dolgachev conjecture}
 \label{sc:Dolgachev}

Let $\bT = (\bCx)^n$ be an algebraic torus and
$\bsM = \Hom(\bT, \bCx)$ be the group of characters.
A convex lattice polytope $\Delta \subset \bsM_\bR = \bsM \otimes \bR$
defines a projective toric variety
$
 X = \Proj S_\Delta
$
where $S_\Delta$ is the monoid ring
of the submonoid of $\bsM \oplus \bN$
consisting of lattice points of the cone
over $\Delta \times \{ 1 \} \subset \bsM_\bR \oplus \bR$.
The \emph{polar polytope} of $\Delta$ is defined by
\begin{align}
 \Deltav = \{ v \in \bsMv \mid \la v, m \ra \ge -1
  \text{ for any } m \in \Delta \}
\end{align}
where $\bsMv = \Hom(\bsM, \bZ)$ is the dual lattice of $\bsM$.
The polytope $\Delta$ is said to be \emph{reflexive}
if the polar dual polytope $\Deltav$ is a lattice polytope, and
the origin is the unique interior lattice point of $\Delta$.
The projective toric variety
associated with $\Deltav$ will be denoted by
$\Xv = \Proj S_{\Deltav}$.
The families $\left| - K_X \right|$
and $\left| - K_{\Xv} \right|$
of anti-canonical hypersurfaces are called a
\emph{Batyrev mirror pair}
\cite{Batyrev_DPMS}.

Assume $n=3$ and take very general members
$Y$ and $\Yv$ of $\left| - K_X \right|$
and $\left| - K_{\Xv} \right|$.
Define $M_\Delta$ as the primitive sublattice of $H^2(Y; \bZ)$
generated by the image of
$\iota^* : H^2(X; \bZ) \to H^2(Y; \bZ)$,
and similarly for $M_{\Deltav} \subset H^*(\Yv, \bZ)$.

For a vector $e$ in a lattice $S$,
the positive integer $\sdiv e$ is defined as
the greatest common divisor of
$(e, f) \in \bZ$
for all $f \in S$.
A primitive isotropic vector $e$ is called {\em $m$-admissible}
if $\sdiv e = m$ and there exists another primitive isotropic vector $f$
such that $(e, f) = m$ and $\sdiv f = m$.

Assume that $M_\Delta^\bot \subset H^2(Y;\bZ)$
has an $m$-admissible vector $e$.
Then one has $M_\Delta^\bot = U(m) \bot \Mv_\Delta$
where $U(m)$ is the lattice generated by $e$ and $f$,
and $\Mv_\Delta$ is the orthogonal complement.

\begin{conjecture}[{Dolgachev
\cite[Conjecture (8.6)]{Dolgachev_MSK3}}]
 \label{cj:Dolgachev}
\ 
\begin{enumerate}
\item
The lattice $M_\Delta^\bot$ contains
a 1-admissible isotropic vector.
\item
There exists a primitive embedding
$
 M_\Deltav \subset \Mv_\Delta.
$
\item
The equality
$
 M_\Deltav = \Mv_\Delta
$
holds if and only if
$
 M_\Delta \cong \Pic Y.
$
\end{enumerate}
\end{conjecture}

\section{Secondary stack}
 \label{sc:secondary}

Let $\Delta$ be a reflexive polytope in $\bsM_\bR$ and
$
 A
  = \lc v_0=0, v_1, \ldots, v_{n+r} \rc
$
be the set of lattice points of $\Delta$.
It gives the \emph{fan sequence}
\begin{align} \label{eq:fan_sequence}
 0 \to \bL \to \bZ^{n+r} \xto{\beta} \bsNv \to 0,
\end{align}
where $\bsNv = \bsM$
and $\bL$ is the kernel of the homomorphism
$
 P \colon \bZ^{n+r} \to \bsNv
$
sending the $i$-th coordinate vector $e_i$ to $v_i$
for $i=1, \ldots, {n+r}$.
We write a basis of $\bL$ as
$\{ c^{(p)} \}_{p=1}^{r}$
where $c^{(p)} = (c_1^{(p)}, \ldots, c_{n+r}^{(p)})$.
By setting $\vtilde_i = (v_i, 1) \in \bsNv \oplus \bZ$
for $i=0, \ldots, {n+r}$,
one obtains a sequence
\begin{align} \label{eq:e_fan_sequence}
 0 \to \bL \to \bZ^A \xto{\betatilde} \bsNv \oplus \bZ \to 0,
\end{align}
where the map
$
 \Ptilde \colon \bZ^A \cong \bZ^{{n+r}+1} \to \bsNv \oplus \bZ
$
sends $e_i$ to $\vtilde_i$
for $i=0, \ldots, {n+r}$,
and an element $c^{(p)} \in \bL$
is mapped to
$$
 \ctilde^{(p)} = (-c_1^{(p)}-\cdots-c_{n+r}^{(p)}, c_1^{(p)}, \ldots, c_{n+r}^{(p)})
  \in \bZ^{{n+r}+1}.
$$
The sequence
\begin{align} \label{eq:divisor_sequence}
 0 \to \bsMv \xto{\beta^\vee} \bZ^{n+r} \to \bL^\vee \to 0
\end{align}
dual to \eqref{eq:fan_sequence} is called
the \emph{divisor sequence}.

Given a polyhedral sudivision
$\Deltaul = \{ \Delta_1, \ldots, \Delta_k \}$ of $\Delta$,
one sets
\begin{align}
 C(\Deltaul) = \lc \psi \in \bR^A \relmid g_\psi \text{ is affine linear
  over each polytope in } \Deltaul \rc,
\end{align}
where $g_\psi : \Delta \to \bR$ is the convex piecewise linear function
associated with $\psi : A \to \bR$.
The cone $C(\Deltaul)$ is invariant
under the additive action
of the space $\Aff(\bsM_\bR)$ of affine linear functions on $\bsM_\bR$.
The quotient cones $C(\Deltaul)/\Aff(\bsM_\bR)$ constitute
a complete fan $\cF(A)$ in $\bR^A/\Aff(\bsM_\bR)$
called the \emph{secondary fan}
\cite{Gelfand-Kapranov-Zelevinsky_DRMD}.
Maximal cones of the secondary fan $\cF(A)$ correspond
to coherent triangulations of the polytope $\Delta$.
A {\em circuit} is an affinely dependent subset
any of whose proper subset is affinely independent
\cite[7.1.B]{Gelfand-Kapranov-Zelevinsky_DRMD}.
Adjacencies of triangulations
come from modifications along circuits
\cite[Theorem 7.2.10]{Gelfand-Kapranov-Zelevinsky_DRMD}.

The elements $\{ \ctilde^{(p)} \}_{p=1}^{r}$
generate one-dimensional cones of the secondary fan $\cF(A)$,
which gives a structure of a stacky fan on $\cF(A)$.
The toric stack $\XFA$ associated with the resulting stacky fan
will be called the \emph{secondary stack}
\cite{Diemer-Katzarkov-Kerr_SGR}.
The dense torus of the secondary stack
can naturally be identified with $\bL_{\bCx}$.
The coarse moduli space of $\XFA$
is the \emph{Chow quotient}
of $\bP(\bC^A)$ by $\bT$
\cite{MR1119943,MR1174606}.

For a face $\Delta'$ of $\Deltaul$
(i.e., a face of some $\Delta_i$ in $\Deltaul$),
set
$$
 C(\Deltaul, \Delta')
  = \lc \psi \in C(\Deltaul) \relmid g_\psi \text{ attains its minimum on } \Delta' \rc.
$$
The cone $C(\Deltaul, \Delta')$ is invariant
under the action of constant functions in $\Aff(\bsM_\bR)$.
The abelian group $\bZ^A/\bZ$ can naturally be identified
with $\bsN_H = \Hom(\bCx, H)$.
The cones $C(\Deltaul, \Delta')/\bR$ constitutes
a complete fan $\cFtilde(A)$ in $\bR^A/\bR$
called the \emph{Lafforgue fan}
\cite{MR1976905,Hacking_CMHA}.
The toric stack associated with Lafforgue fan
equipped with a stacky structure
is called the \emph{Lafforgue stack}
\cite{Diemer-Katzarkov-Kerr_SGR}.
The natural homomorphism
$
 \bZ^A/\bZ \cong \bsN_H \to \bZ^A/\Aff_\bZ(\bsM) \cong \bsN
$
defines a morphism
$
 \cFtilde(A) \to \cF(A)
$
of fans,
which induces a torus-equivariant morphism
$
 \varphi_X \colon X_{\cFtilde(A)} \to \XFA
$
of toric stacks.

The Laurent polynomial
\begin{align}
 W = \sum_{i=0}^{{n+r}} a_i x_1^{v_{i1}} \cdots x_n^{v_{in}}
\end{align}
gives a section of $(\varphi_X)_* \lb \cO_{X(\Ftilde(A))}(1) \rb$,
where $\cO_{\XFtildeA}(1)$ is the line bundle
which restricts to the anti-canonical bundle $\scO(-K_X)$
on a general fiber.
The zero of $W$ gives a family
$\varphi_Y : \fY_A \to \XFA$
of hypersurfaces.

The \emph{discriminantal variety} $\bsnabla_A \subset \bC^A$
is the closure of all $a \in \bC^A$ such that
there exists $x \in (\bCx)^n$ satisfying
\begin{align}
 W
  = \frac{\partial W}{\partial x_1}
  = \cdots
  = \frac{\partial W}{\partial x_n}
  =0.
\end{align}
The \emph{$A$-discriminant}
is the irreducible polynomial
$\bsDelta_A \in \bZ[a_0, \ldots, a_{n+r}]$
vanishing on $\bsnabla_A$.
The Newton polytope of $\bsDelta_A$ is
the \emph{secondary polytope} of $A$.
The normal fan to the secondary polytope
is the secondary fan $\scF(\Sigma)$.
Since $\bsnabla_A \cap (\bCx)^A$
is invariant under the action of $\bT \times \bCx$,
it descends to a hypersurface
$\bsnablabar_A \subset \bL_{\bCx}$
called the \emph{reduced $A$-discriminantal variety}.
The \emph{Horn-Kapranov uniformization}
\cite{MR1510591,MR1109634}
is the rational map
$
 h \colon \bP^{r} \to \bL_{\bCx},
$
$
 \lambda = [\lambda_1:\cdots:\lambda_{r+1}]
  \mapsto (\Phi_1(\lambda_1), \ldots, \Phi_{r}(\lambda)) 
$
where
\begin{align} \label{eq:Horn-Kapranov}
 \Phi_q(\lambda) &= \prod_{j=1}^n
  \lb \sum_{p=1}^{r} c_j^{(p)} \lambda_p \rb^{c_j^{(q)}}.
\end{align}
The image of $h$ is the reduced $A$-discriminantal variety.
The restriction of the family $\varphi_Y$
to the complement
$
 \XFA^\reg = \bL_{\bCx} \setminus \bsnabla_A
$
of the discriminantal variety
will be denoted by
\begin{align}
 \varphi_Y^\reg : \fY_A^\reg \to \XFA^\reg.
\end{align}

\section{Mirror symmetry and monodormy}
 \label{sc:monodromy}

Hodge theory gives
an integral variation
$
 (\HBZ^\vc, \nabla^B, \scrF_B^\bullet, Q_B)
$
of polarized pure Hodge structures
\cite[Definitions 6.5 and 6.7]{Iritani_QCP}, where
\begin{itemize}
 \item
$\HBZ^\vc$ is the local system on $\XFA^\reg$
whose fiber over $[a] \in \XFA^\reg$
is the sublattice of $H^{n-1}(Y_a, \bZ)$
generated by vanishing cycles,
 \item
$\nabla^B$ is the Gauss-Manin connection
on $\HBZ^\vc \otimes \cO_{\XFA^\reg}$,
 \item
$\scrF_B^\bullet$ is the Hodge filtration, and
 \item
$Q_B$ is the polarization.
\end{itemize}
On the mirror side,
let
\begin{align}
 H^\bullet_\amb(\Yv; \bC)
  = \Image(\iota^* : H^\bullet(\Xv; \bC) \to H^\bullet(\Yv; \bC))
\end{align}
be the subspace of $H^\bullet(\Yv; \bC)$
coming from the cohomology classes of the ambient toric variety,
and set
\begin{align}
 U = \lc \sigma = \beta + \sqrt{-1} \omega \in H^2_\amb(\Yv; \bC) \relmid
   \la \omega, d \ra \gg 0 \text{ for any non-zero }
   d \in \NE(\Yv) \rc,
\end{align}
where $\NE(\Yv)$ is the semigroup of effective curves.
Let $\{ p_i \}_{i=1}^r$ be an integral basis
of the nef cone of $\Yv$, and
$(\sigma^i)_{i=1}^r$ be the dual coordinate on $H^2_\amb(Y; \bC)$.
The {\em ambient A-model VHS}
$(\HAZ^\amb, {\nabla^A}, {\scrF_A}^\bullet, Q_A)$
consists of a local system $\HAZ^\amb$ on $U$,
the Dubrovin connection
\begin{align}
 \nabla^A = d + \sum_{i=1}^r (p_i \circ_\sigma) \, d \sigma^i
  \colon \scrH_A \to \scrH_A \otimes \Omega_{U}^1,
\end{align}
on $\scrH_A = H_\amb^\bullet(\Yv;\bC) \otimes \cO_U$,
the Hodge filtration
$$
 \scrF_A^p = H_\amb^{4-2p}(\Yv;\bC) \otimes \scO_U,
$$
and the Poincar\'{e} pairing
$$
 Q_A : \scrH_A \otimes \scrH_A \to \scO_U.
$$
See \cite[Definition 6.2]{Iritani_QCP}
for details.
The fiber of the local system $\HAZ^\amb$ is isomorphic
to the subgroup $\cN_\amb(\Yv)$
of the numerical Grothendieck group $\cN(\Yv)$
generated by classes pulled-back from $\cN(\Xv)$.
The numerical Grothendieck group $\cN(\Yv)$
is the quotient of the Grothendieck group $K(\Yv)$
by the radical of the Euler form
\begin{align}
 \chi(\cE, \cF)
  &:= \sum_{i=0}^2 (-1)^i \dim \Ext^i(\cE, \cF).
\end{align}
Riemann-Roch theorem shows that
$\cN(\Yv)$ is isomorphic to the image of $K(\Yv)$ by the map
\begin{align}
 v : K(\Yv) \to H^*(\Yv, \bQ), \quad
  \cE \mapsto \ch(\cE) \cup \Gammahat_\Yv,
\end{align}
where the $\Gammahat$-class
is a square root of the Todd class
\cite{Iritani_ISQCMSTO}.
The polarization $Q_A$ is mapped to minus the Euler form
under this isomorphism.

Let $u_i \in H^2(X; \bZ)$ be the Poincar\'{e} dual
of the toric divisor
corresponding to the one-dimensional cone
$\bR \cdot v_i \in \Sigma$ and
$v = u_1 + \cdots + u_m$
be the anticanonical class.
Givental's {\em $I$-function} is defined as the series
\begin{align}
 I_{X, Y}(q, z) = e^{p \log q / z}
  \sum_{d \in \NE(X)} q^d \,
  \frac{
   \prod_{k=-\infty}^{\la d, v \ra} (v + k z)
   \prod_{j=1}^m \prod_{k=-\infty}^0
    (u_j + k z)}
  {\prod_{k=-\infty}^0
    (v + k z)
   \prod_{j=1}^m \prod_{k=-\infty}^{\la d, u_j \ra}
    (u_j + k z)},
\end{align}
which is a map from $\Uv$
to the classical cohomology ring $H^\bullet(X; \bC[z^{-1}])$.
When $Y$ is a K3 surface,
Givental's {\em $J$-function} is given by
\begin{align}
 J_Y(\tau, z)
  = \exp(\tau/z).
\end{align}
If we write
\begin{align}
 I_{X, Y}(q, z) = F(q) + \frac{G(q)}{z} + \frac{H(q)}{z^2} + O(z^{-3}),
\end{align}
then Givental's mirror theorem
\cite{Givental_EGWI,
Givental_MTTCI,
Coates-Givental}
states that
\begin{align}
 \Euler(\omega_X^{-1}) \cup I_{X, Y}(q, z)
  = F(q) \cdot \iota_* J_Y(\varsigma(q), z)
\end{align}
where $\Euler(\omega_X^{-1}) \in H^2(X; \bZ)$
is the Euler class of the anticanonical bundle of $X$,
and the {\em mirror map}
$
 \varsigma(q) : \Uv \to H^2_\amb(Y; \bC)
$
is defined by
\begin{align}
 \varsigma(q) = \iota^* \lb \frac{G(q)}{F(q)} \rb.
\end{align}
The relation between $\tau = \varsigma(q)$ and
$\sigma = \beta + \sqrt{-1} \omega$ is given by
$\tau = 2 \pi \sqrt{-1} \sigma$,
so that $\Im(\sigma) \gg 0$ corresponds to $\exp(\tau) \sim 0$.
The functions $F(q)$, $G(q)$ and $H(q)$ satisfy
the Gelfand--Kapranov--Zelevinsky hypergeometric differential equations,
and give periods for the B-model VHS.

\begin{theorem}[{Iritani \cite[Theorem 6.9]{Iritani_QCP}}]
 \label{th:Iritani}
There is an isomorphism
\begin{align}
 \Mir_\scY : \varsigma^*
  (\HAZ^\amb, \nabla^A, \scrF_A^\bullet, Q_A)
  \simto (\HBZ^\vc, \nabla^B, \scrF_B^\bullet, Q_B)
\end{align}
of integral variations of pure and polarized Hodge structures.
\end{theorem}

The main step in the proof of \pref{th:Iritani}
is \cite[Theorem 5.7]{Iritani_QCP},
which relates periods of A-model VHS
and those of B-model VHS.
In the proof of \cite[Theorem 5.7]{Iritani_QCP},
Iritani shows that the monodromy of the B-model VHS
along a small loop
$q_i \mapsto e^{2 \pi \sqrt{-1}} q_i$
is mapped to the isometry
\begin{align} \label{eq:mon0_1}
 (-) \otimes \iota^*(\cL_i^\vee) \colon \cN(\Yv) \to \cN(\Yv)
\end{align}
where $\cL_i$ is the line bundle on $\Xv$
with $c_1(\cL_i) = p_i$.
Note that this isometry comes from
an autoequivalence of $D^b \coh \Yv$.
The relation between monodromy of the periods
and autoequivalences of the derived category of coherent sheaves
on the mirror manifold
goes back to \cite{Kontsevich_ENS98,Horja_DCAMS}.

When $\Yv$ is a K3 surface,
then the numerical Grothendieck group
$\cN(\Yv)$ is the direct sum
$\bZ [\cO_\Yv] \oplus \Pic(\Yv) \oplus \bZ [\cO_p]$
of the Picard group $\Pic(\Yv)$
and the free module generated by
the classes of the structure sheaf
and a skyscraper sheaf.
The embedding of the Picard group
to the numerical Grothendieck group is given by
$[\cO_\Yv(D)] \mapsto [\cO_D]$.
\begin{comment}
\begin{lemma}
\begin{align}
 D \cdot E
  &= - \chi(\cO_D, \cO_E).
\end{align}
\end{lemma}
\begin{proof}
The exact sequence
\begin{align} \label{eq:ex1}
 0 \to \cO(-D) \to \cO \to \cO_D \to 0,
\end{align}
gives the relation
\begin{align}
 [\cO] = [\cO(-D)] + [\cO_D].
\end{align}
The dual of \eqref{eq:ex1} gives the identity
\begin{align}
 \cO_D^\vee
  = [\cO] - [\cO(D)]
  = - [\cO_D(D)]
  = - [\cO_D(D \cdot D)]
  = - [\cO_D] - (D \cdot D) [\cO_p].
\end{align}
The definition of intersection numbers gives
\begin{align}
 D \cdot E
  &= \chi(\cO_D^\vee, \cO_E) \\
  &= -(\cO_D^\vee, \cO_E) \\
  &= -(- [\cO_D] - (D \cdot D) [\cO_p], \cO_E) \\
  &= (\cO_D, \cO_E).
\end{align}
so that
\begin{align}
 (\cO_D, \cO_E) = D \cdot E.
\end{align}
\end{proof}
\end{comment}
The lattice structure is given by
\begin{gather}
 (\cO_\Yv, \cO_\Yv) = -2, \
 (\cO_\Yv, \cO_D) = -\chi(\cO_D), \ 
 (\cO_\Yv, \cO_p) = -1, \\
 (\cO_D, \cO_E) = D \cdot E, \ 
 (\cO_D, \cO_p) = -\chi(\cO_p, \cO_p) = 0,
\end{gather}
where $\cO_{\Yv}$ is the structure sheaf,
$\cO_D$ is the structure sheaf of a divisor $D$,
and $\cO_p$ is a skyscraper sheaf.
If $D$ is a smooth curve of genus $g$,
then one has
$
 \chi(\cO_D) := \dim H^0(\cO_D) - \dim H^1(\cO_D) = 1 - g
$
and
$
 D \cdot D = 2g-2.
$
The action of $(-) \otimes \cO_\Yv(-D) : K(\Yv) \to K(\Yv)$
is given by
\begin{align}
 [\cO_\Yv] &\mapsto [\cO(-D)] = [\cO_\Yv] - [\cO_D], \\
 [\cO_E] &\mapsto [\cO_E(-D)]
  = [\cO_E] - (D \cdot E) [\cO_p], \\
 [\cO_p] &\mapsto [\cO_p].
\end{align}

\section{The family associated with $A_0$}
 \label{sc:A0}

\subsection{The secondary fan}

Let
$
 \Delta_0 = \Conv \{ v_1, v_2, v_3, v_4, v_5 \}
$
be the reflexive polytope
with 5 vertices
in \pref{fg:P0};
\begin{align}
 \beta_0 &=
\begin{pmatrix}
 v_1 & v_2 & v_3 & v_4 & v_5
\end{pmatrix}
 =
\begin{pmatrix}
 1 & 0 & 0 & 0 & -1 \\
 0 & 1 & 0 & 0 & -1 \\
 0 & 0 & 1 & -1 & -2 \\
\end{pmatrix}.
\end{align}
The homomorphism $\bZ^{n+r} \to \bL^\vee$
in the divisor sequence
\eqref{eq:divisor_sequence}
is represented by the matrix
\begin{align}
\begin{pmatrix}
 -2 & 0 & 0 & 1 & 1 & 0 \\
 -5 & 1 & 1 & 2 & 0 & 1
\end{pmatrix},
\end{align}
and the secondary fan is given in \pref{fg:sfan0}. 
\pref{fg:sfan2} shows maximal cones
of the secondary fan $\cF(A_0)$.
The cone $I$ corresponds to the large complex structure limit point.
The corresponding coherent triangulation is given by
%
\begin{comment}
\begin{shaded}
\begin{align*}
\bordermatrix{
 & a_0 & a_1 & a_2 & a_3 & a_4 & a_5 \cr
 & -2 & 0 & 0 & 1 & 1 & 0 \cr
 & -5 & 1 & 1 & 2 & 0 & 1
}
\end{align*}
\end{shaded}
\end{comment}
%
\begin{figure}[ht]
\centering
\input{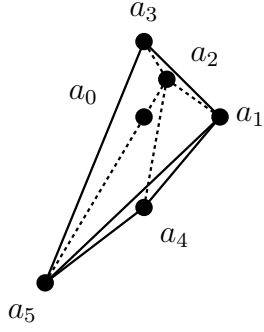}
\caption{The polytope $\Delta_0$}
\label{fg:P0}
\end{figure}
\begin{figure}[ht]
\begin{minipage}{.5 \linewidth}
\centering
\input{sfan0.pst}
\caption{The secondary fan $\cF(A_0)$}
\label{fg:sfan0}
\end{minipage}
\begin{minipage}{.5 \linewidth}
\centering
\input{sfan2.pst}
\caption{Maximal cones of $\cF(A_0)$}
\label{fg:sfan2}
\end{minipage}
\end{figure}
\begin{align*}
 \{ a_0, a_1, a_2, a_3 \}
  \cup \{ a_0, a_2, a_3, a_5 \}
  \cup \{ a_0, a_1, a_3, a_5 \}
  \cup \{ a_0, a_1, a_2, a_4 \}
  \cup \{ a_0, a_2, a_4, a_5 \}
  \cup \{ a_0, a_1, a_4, a_5 \}.
\end{align*}
In the cone $\II$,
it is given by
\begin{align*}
 \{ a_1, a_2, a_3, a_4 \}
  \cup \{ a_2, a_3, a_4, a_5 \}
  \cup \{ a_1, a_3, a_4, a_5 \},
\end{align*}
which is the `vertical' triangulation.
In the cone $\III$,
it is given by 
\begin{align*}
 \{ a_1, a_2, a_3, a_5 \}
  \cup \{ a_1, a_2, a_4, a_5 \},
\end{align*}
which is the `horizontal' triangulation.
In the cone $IV$,
it is given by
\begin{align*}
 \{ a_0, a_1, a_2, a_3 \}
  \cup \{ a_0, a_1, a_3, a_5 \}
  \cup \{ a_0, a_1, a_2, a_5 \}
  \cup \{ a_0, a_2, a_3, a_5 \}
  \cup \{ a_1, a_2, a_4, a_5 \}.
\end{align*}
$I$ is obtained from $\II$ by the subdivision at $a_0$, i.e.,
the modification along the circuit $\{a_0, a_3, a_4 \}$.
The family over the corresponding divisor
is the union of three rational surfaces
intersecting along three rational curves.
All these three curves passes through two points,
so that the dual graph of the intersection is the division of $S^2$
into two triangles.
$\II$ and $\III$ are related
by the modification along the circuit $\{ a_1, a_2, a_3, a_4, a_5 \}$.
The family parametrized by the corresponding divisor
is a family of K3 surfaces.
$\III$ and $IV$ are related
by the subdivision of $\{ a_1, a_2, a_3, a_5 \}$,
i.e., the modification along the circuit $\{ a_0, a_1, a_2, a_3, a_5 \}$.
Note that one has
\begin{align*}
 \vol \{ a_1, a_2, a_3, a_5 \} &= 5, \\
 \vol \{ a_0, a_1, a_2, a_5 \} &= 2.
\end{align*}
The corresponding family is the family of K3 surfaces
associated with the polytope
$\{ a_1, a_2, a_3, a_5 \}$.
$IV$ and $I$ are related by the modification
along the circuit $\{ a_0, a_1, a_2, a_4, a_5 \}$.
The corresponding family is the union of four rational surfaces
whose dual graph is a tetrahedron.

\begin{comment}
\begin{shaded}
\begin{proof}
In the exact sequence
\begin{align*}
 0 \to \Aff(\bsM_\bR) \xto{\Ptilde^T} \bR^A \to \bR^A/\Aff(\bsM_\bR) \cong \bR^2 \to 0,
\end{align*}
one has
\begin{align*}
 (0,\lambda_1,0,0,\lambda_4,0) \mapsto (\lambda_4, \lambda_1),
\end{align*}
so that
the regular polyhedral decomposition
corresponding to the point $(\lambda_4, \lambda_1) \in \bR^2$
is associated with the function $g_\psi : \bR^2 \to \bR$
coming from $\psi : A \to \bR$ defined by
\begin{align*}
 g(v_0) &= 0, \\
 g(v_1) &= \lambda_1, \\
 g(v_2) &= 0, \\
 g(v_3) &= 0, \\
 g(v_4) &= \lambda_4, \\
 g(v_5) &= 0.
\end{align*}
\end{proof}
\end{shaded}
\begin{figure}[ht]
\centering
\input{sfan1.pst}
\caption{The secondary fan for $P_0$}
\label{fg:sfan1}
\end{figure}
\end{comment}

\subsection{The period domain}

The Picard lattice and the transcendental lattice
of a very general member of the family
$\fY_{A_0} \to \XFzero$ of K3 surfaces
associated with $A_0$ is given as follows:

\begin{theorem}[{\cite[Theorem 3.1]{MR2962398}}]
 \label{th:Nagano0}
The Picard lattice and the transcendental lattice
of a very general member of the family
$\fY_{A_0} \to \XFzero$ are given by
\begin{align}
 M_0 &= E_8 \bot E_8 \bot \sqmat{2}{1}{1}{-2},
\qquad
 T_0 = U \bot \sqmat{2}{1}{1}{-2}.
\end{align}
\end{theorem}
The moduli space of $M_0$-polarized K3 surfaces can be identified
with a Hilbert modular surface as follows.
Let $\cO$ be the ring of integers of the real quadratic field $\bQ(\sqrt{5})$.
The Hilbert modular group $\PSL_2(\cO)$ acts on the product
$\bH \times \bH$ of the upper half planes by
\begin{align}
\begin{pmatrix}
 \alpha & \beta \\
 \gamma & \delta
\end{pmatrix}
 \colon
 (z_1, z_2) \mapsto
  \lb
   \frac{\alpha z_1+\beta}{\gamma z_1+\delta},
   \frac{\alpha' z_1+\beta'}{\gamma' z_1+\delta'}
  \rb,
\end{align}
where
$(-)'$
is the conjugation in $\bQ(\sqrt{5})$.
One has
\begin{align}
 \begin{pmatrix}
 2 & 1 \\
 1 & -2
\end{pmatrix}
 &= W U W^T, \quad
 W =
\begin{pmatrix}
 1 & 1 \\
 -\varepsilon^{-1} & \varepsilon
\end{pmatrix}, \quad
 \varepsilon = (1+\sqrt{5})/2,
\end{align}
so that
\begin{align}
 \bH \times \bH \to \cD^+, \quad
 (z_1, z_2) \mapsto \lb I_2 \oplus (W^T)^{-1} \rb
\begin{pmatrix}
 z_1 z_2 \\
 -1 \\
 z_1 \\
 z_2
\end{pmatrix}
\end{align}
gives a biholomorphic map.
The orthogonal group $\PO^+(T_0)$ is generated
by the Hilbert modular group
$\PSL_2(\cO)$ and the permutation
\begin{align}
 \tau : \bH \times \bH \to \bH \times \bH, \quad
  (z_1, z_2) \mapsto (z_2, z_1)
\end{align}
under this identification.
The symmetric Hilbert modular surface
$
 \bH \times \bH / \la \PSL_2(\cO), \tau \ra
$
is studied in detail
by Hirzebruch \cite{MR0480355}
(cf. also \cite{MR1025329}).
The graded ring $\fM = \bigoplus_{n=0}^\infty \fM_n$
of symmetric Hilbert modular forms
is generated by forms $\fA$, $\fB$, $\fC$, $\fD$
of weights 2, 6, 10, 15
with one relation of degree 30;
\begin{gather}
 \fM = \bC[\fA,\fB,\fC,\fD] / \lb
 144 \fD^2 - \bsDelta(\fA, \fB, \fC) \rb, \\
 \bsDelta(\fA, \fB, \fC) =
  -1728 \fB^5+720 \fA \fB^3 \fC-80 \fA^2 \fB \fC^2
  +64 \fA^3(5 \fB^2-\fA \fC)^2+\fC^3.
\end{gather}
The stack $\cMbar = \bProj \fM=[(\Spec \fM \setminus \bszero)/\bCx]$
is a hypersurface of degree 30
in the weighted projective space $\bP(2,6,10,15)$.
It is obtained from the weighted projective plane
$\bP(1,3,5) = \bProj \bC[\fA,\fB,\fC]$
by the root construction
along the divisor defined by $\bsDelta$.
The cusp consists of one point
$[\fA:\fB:\fC]=[1:0:0]$.

\begin{comment}
\begin{shaded}
The family
\begin{align} \label{eq:Weierstrass_0}
 z^2 = x^3 - 4(4y^3-5 \fA y^2) x^2 + 20 \fB x y^3 + \fC y^4
\end{align}
gives a birational model
for the family of lattice-polarized K3 surfaces
\cite{MR3160602}.
Two members of this family is isomorphic
as a lattice polarized K3 surface
if and only if there is $\alpha \in \bCx$
such that
$(\fA', \fB', \fC') = (\alpha \fA, \alpha^3 \fB, \alpha^5 \fC)$.
\end{shaded}
\end{comment}

\subsection{The period map}

The space of Laurent polynomials is given by
\begin{align} \label{eq:Laurent0}
 \bC^{A_0} =
  \lc W = a_0+a_1x+a_2y+a_3z+\frac{a_4}{z}+\frac{a_5}{xyz^2} \rc,
\end{align}
and the discriminant is given by
\begin{align}
\begin{split}
 \bsDelta = a_4^2 a_0^6 &+4 a_1 a_2 a_5 a_0^5-12 a_3 a_4^3
   a_0^4-50 a_1 a_2 a_3 a_4 a_5 a_0^3 \\
   &+ 48 a_3^2 a_4^4 a_0^2
   +1000 a_1 a_2 a_3^2 a_4^2 a_5
   a_0-64 a_3^3 a_4^5+3125 a_1^2 a_2^2 a_3^2
   a_5^2.
\end{split}
\end{align}
\begin{comment}
\begin{shaded}
or
\begin{align} \label{eq:disc0-1}
 64 \lambda ^5-48 \lambda ^4+12 \lambda ^3-1000
   \lambda ^2 \mu -\lambda ^2+50 \lambda  \mu
   -3125 \mu ^2-4 \mu.
\end{align}
One can check by direct calculation
that this coincides with
\begin{align}
 1728 X^5 - 720 X^3 Y Z + 80 X Y^2 Z^2 - 64(5 X^2-Y Z)^2 Z^3 - Y^3
\end{align}
up to a multiplicative factor.
By substituting $\mu=0$ into \eqref{eq:disc0-1},
one obtains
\begin{align}
 64 \lambda ^5-48 \lambda ^4+12 \lambda ^3-\lambda ^2
 = \lambda^2(4 \lambda-1)^3.
\end{align}
This shows that the point $(\lambda, \mu) = (1/4,0)$ is on the discriminant.
We will write
\begin{align}
 \lambda &= \lambda, \\
 \mu &= \mu.
\end{align}
\end{shaded}
\end{comment}
%
The dense torus $\bL^\vee_{\bCx}$
of the secondary stack $\XFzero$ can be written as
$\Spec \bC[\lambda^{\pm 1}, \mu^{\pm 1}]$ where
\begin{align}
 \lambda &= \frac{a_2 a_4}{a_0^2}, \\
 \mu &= \frac{a_1 a_2 a_3^2 a_5}{a_0^5}.
\end{align}
The period map
\begin{align}
 \XFzero \dashrightarrow \bP(1,3,5)
\end{align}
for the family $\cY_0$ is computed
in \cite[Theorem 6.2]{MR2962398} as
\begin{align}
 (\lambda, \mu)
  \mapsto \ld 1 : \frac{25 \mu}{2 (\lambda-1/4)^3} :
   - \frac{3125 \mu^2}{(\lambda-1/4)^5} \rd.
\end{align}
Since $\XFzero$ is a weighted blow-up
of $\bP(1,2,5)$,
the period map induces a rational map
\begin{align}
 \bP(1,2,5) \dashrightarrow \bP(1,3,5),
\end{align}
which is given by
\begin{align}
 [\nu:\lambda:\mu]
  \mapsto \ld \lambda-\nu^2/4:
   \frac{25}{2} \nu \mu:
   -3125\mu^2 \rd.
\end{align}
This map is not defined at $[\nu:\lambda:\mu] = [1:1/4:0]$.
The weighted blow-up of weight $(1,3)$
along the ideal $(\lambda-\nu^2/4, \mu)$
eliminates the indeterminacy,
and the resulting morphism is a weighted blow-down
of weight $(1,2)$
which contracts the strict transform of the divisor
$\{ \mu=0 \} \subset \bP(1,2,5)$
to the point $[1:0:0] \in \bP(1,3,5)$.
Since the secondary stack $\XFzero$ is obtained from $\bP(1,2,5)$
by a weighted blow-up of weight $(1,2)$ at one point,
the indeterminacy of the rational map
$\Pi \colon \XFzero \dashrightarrow \bP(1,3,5)$
can also be eliminated by a weighted blow-up
$\XtildeFzero \to \XFzero$ of weight $(1,3)$,
and the resulting morphism
$\Pitilde \colon \XtildeFzero \to \bP(1,3,5)$
is an iteration of weighted blow-ups
of weight $(1,2)$.
A schematic picture of the blow-ups
of the parameter space
is shown in \pref{fg:fans}.

\begin{figure}[ht]
\centering
\input{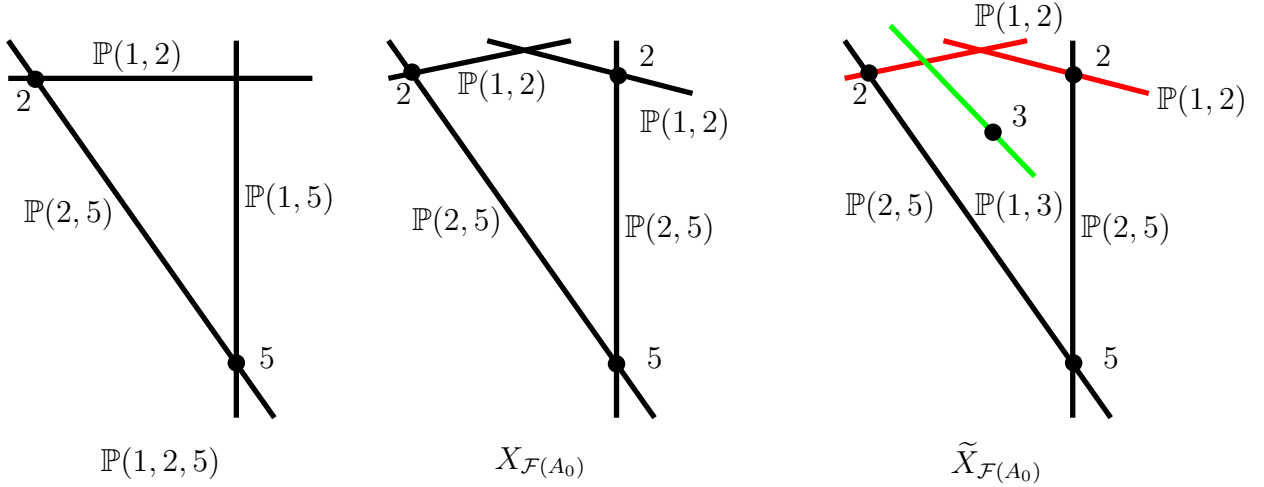}
\caption{The blow-ups of parameter spaces}
\label{fg:fans}
\end{figure}

\subsection{Mirror symmetry and monodromy}

The polar dual polytope of $\Delta_0$ is given by
\begin{align*}
 \Deltav_0 &= \Conv \lc
  (0,-1,1),(4,-1,-1),(-1,-1,-1),(-1,-1,1),((-1,4,-1),(-1,0,1)
 \rc.
\end{align*}
The associated toric variety $\Xv$
is a toric Fano manifold,
which is a $\bP^1$-bundle
$\bP(\cO_{\bP^2} \oplus \cO_{\bP^2}(2))$ over $\bP^2$.
The following theorem of Mo{\u\i}{\v{s}}ezon shows that
the Picard lattice of a very general member $\Yv$
of the family $\fYv$ is generated
by the restrictions of the toric divisors
of the ambient space:

\begin{theorem}[{\cite[Theorem 7.5]{MR0213351}}]
Let $V$ be a smooth projective 3-fold and
$\iota : E \hookrightarrow V$ be a very general hyperplane section.
Then the map $\iota^* : \Pic(E) \to \Pic(V)$ is surjective
if and only if one of the following holds:
\begin{enumerate}
 \item
The Betti numbers satisfy
$b_2(V) = b_2(E)$.
 \item
The Hodge numbers satisfy
$h^{2,0}(E) > h^{2,0}(V)$.
\end{enumerate}
\end{theorem}

It follows that
\begin{align}
 M_{\Delta_0} &= \Pic Y_0
  = E_8 \bot E_8 \bot \sqmat{2}{1}{1}{-2}, \\
 M_{\Delta_0}^\bot &= U \bot \sqmat{2}{1}{1}{-2}, \\
 M_{\Deltav_0}
  &= \Pic \Yv_0
  = \sqmat{2}{1}{1}{-2}, \\
 M_{\Deltav_0}^\bot
  &= U \bot E_8 \bot E_8 \bot \sqmat{2}{1}{1}{-2},
\end{align}
so that \pref{cj:Dolgachev} holds in this case.

The nef cone of $\Xv$ is generated by
$2 D_1+D_4$ and
$D_1$,
where $D_1$ and $D_4$ are toric divisors
associated with one-dimensional cones
generated by $v_1$ and $v_4$.
Let $E_1$ and $E_2$ be the restrictions of $2 D_1+D_4$
and $D_1$ to $\Yv$ respectively.
The corresponding coordinates $(q_1, q_2)$
of $H^2_\amb(\Yv, \bCx)$
near the large radius limit
can be identified with the coordinates
$(\lambda, \mu)$
of the dense torus
$\bL^\vee_{\bCx} \subset \XFzero$
of the secondary stack by
\begin{align}
 q_1 &= \lambda, \\
 q_2 &= \frac{\mu}{\lambda^2}.
\end{align}
The algebraic lattice of $\Yv$ can be written as
$
 \cN(\Yv) \cong U \bot N,
$
where the hyperbolic plane
$U = \bZ e \oplus \bZ f$
is generated by
$
 e = [\cO_p]
$
and
$
 f = - [\cO_\Yv]-[\cO_{E_1}]+3[\cO_{E_2}].
$
The orthogonal complement
\begin{align}
 N \cong
\begin{pmatrix}
 10 & 5 \\
 5 & 2
\end{pmatrix}
\end{align}
is generated by
$
 \lc e_1 = [\cO_{E_1}], \ 
 e_2 = [\cO_{E_2}] \rc
$
and isometric to the N\'{e}ron-Severi lattice
$\NS(\Yv)$.
The orthogonal group of $N$ is generated by two elements;
\begin{align}
 O^+(N) = \la g_1, g_2 \ra,
\quad
 g_1 = 
\begin{pmatrix}
 4 & 1 \\
 -5 & -1
\end{pmatrix}, \ 
 g_2 = 
\begin{pmatrix}
 1 & 1 \\
 0 & -1
\end{pmatrix} .
\end{align}
The monodromies $T_1$ and $T_2$ for
$
 (q_1, q_2) \mapsto (e^{2 \pi \sqrt{-1}} q_1, q_2)
$
and
$
 (q_1, q_2) \mapsto (q_1, e^{2 \pi \sqrt{-1}} q_2)
$
are given by
$
 \cO(-E_1) \otimes (-)
$
and
$
 \cO(-E_2) \otimes (-)
$
respectively.
A direct calculation shows
\begin{align}
 T_1 = \varphi_{e, e_1}, \\
 T_2 = \varphi_{e, e_2},
\end{align}
where $\varphi_{e, \bullet} \colon N \hookrightarrow O(\cN(\Yv))$
is the embedding in \eqref{eq:embedding}.

\subsection{Proof of \pref{th:main0}}

Let $\Sigma$ be the fan in $N_\bR$
whose one-dimensional cones are generated by
the orbit of $e_1$ and $e_2$
under the action of $O^+(N)$
shown in \pref{fg:toroidal_fan}.
The associated toric variety $X_\Sigma$
shown in \pref{fg:toroidal_translation}
has a natural action of $O^+(N)$,
and the toroidal compactification
is obtained by replacing a neighborhood of the cusp
with the neighborhood of the origin in $X_\Sigma$.
The quotient $X_\Sigma/O^+(N)$
is obtained by first taking quotient by the infinite cyclic subgroup
$C_1 \lhd O^+(N)$ generated by $g_1$,
and then by the cyclic group $C_2 = O^+(N)/C_1$ of order two
generated by $[g_2]$.
The action of $g_1$ on $\Sigma$
is a `translation'
sending a one-dimensional cone
to the one which is next next to it.
The quotient of $X_\Sigma$ by this action
gives a configuration of a $(-1)$-curve and a $(-5)$-curve
shown in \pref{fg:toroidal_translation2}.
The quotient group $C_2$ acts on $X_\Sigma/C_1$
by flipping along the horizontal line in \pref{fg:toroidal_flip1},
and one obtains a chain of $\bP(1,2)$
shown in \pref{fg:toroidal_flip2}
as the quotient.
The normal bundles of these curves are
$\cO_{\bP(1,2)}(-1)$ and $\cO_{\bP(1,2)}(-5)$,
which are precisely the ones
that one obtains by performing iterated weighted blow-up
of weight $(1,2)$.
By contracting these curves,
one obtains the cusp,
which is the smooth point $[1:0:0]$
in $\bP(1,3,5)$.

\begin{figure}[ht]
\begin{minipage}{.35 \linewidth}
\centering
\input{toroidal_fan.pst}
\caption{The fan $\Sigma$ in $N_\bR$}
\label{fg:toroidal_fan}
\end{minipage}
\begin{minipage}{.33 \linewidth}
\centering
\input{toroidal_translation.pst}
\caption{The variety $X_\Sigma$}
\label{fg:toroidal_translation}
\end{minipage}
\begin{minipage}{.3 \linewidth}
\centering
\input{toroidal_translation2.pst}
\caption{The variety $X_\Sigma/C_1$}
\label{fg:toroidal_translation2}
\end{minipage}
\begin{minipage}{.45 \linewidth}
\centering
\ \\
\ \hspace{15mm}\\
\input{toroidal_flip1.pst}
\caption{The action of $C_2$}
\label{fg:toroidal_flip1}
\end{minipage}
\begin{minipage}{.45 \linewidth}
\centering
\ \\
\ \hspace{15mm}\\
\input{toroidal_flip2.pst}
\ \\
\ \hspace{25mm}\\
\caption{The variety $X_\Sigma/O^+(N)$}
\label{fg:toroidal_flip2}
\end{minipage}\end{figure}

\section{The family associated with $A_1$}
 \label{sc:A1}

\subsection{The secondary fan}

\begin{figure}[ht]
\centering
\input{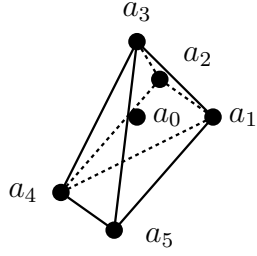}
\caption{The polytope $P_1$}
\label{fg:P1}
\end{figure}

\begin{figure}[ht]
\begin{minipage}{.45 \linewidth}
\centering
\input{sfan_1.pst}
\caption{The secondary fan $\cF(A_1)$}
\label{fg:sfan_1}
\end{minipage}
\begin{minipage}{.45 \linewidth}
\centering
\input{sfan_1_2.pst}
\caption{Maximal cones of $\cF(A_1)$}
\label{fg:sfan_1_2}
\end{minipage}
\end{figure}


Let $\Delta_1$ be the reflexive polytope
with 5 vertices in \pref{fg:P1};
\begin{align}
 \beta_1 &=
\begin{pmatrix}
 v_1 & v_2 & v_3 & v_4 & v_5
\end{pmatrix}
 =
\begin{pmatrix}
 1 & 0 & 0 & -1 & 0 \\
 0 & 1 & 0 & 0 & -1 \\
 0 & 0 & 1 & -1 & -1 \\
\end{pmatrix}.
\end{align}
The homomorphism $\bZ^{n+r} \to \bL^\vee$
in the divisor sequence
\eqref{eq:divisor_sequence}
is represented by the matrix
\begin{align*}
\begin{pmatrix}
 -3 & 1 & 0 & 1 & 1 & 0 \\
 -3 & 0 & 1 & 1 & 0 & 1
\end{pmatrix}.
\end{align*}
The secondary fan $\cF(A_1)$ is given in \pref{fg:sfan_1},
and the secondary stack $\XFone$ is $\bP(1,3,3)$ blown-up
at one point.
%
\begin{comment}
\begin{shaded}
The triangulation associated with
$(0, 0, 0, 0, \lambda_1, \lambda_2)$.
\end{shaded}
\end{comment}
%
\pref{fg:sfan_1_2} shows maximal cones
of the secondary fan $\cF(A_1)$.
The cone $I$ corresponds to the triangulation
\begin{align}
\begin{split}
 \{ a_0,a_2,a_3,a_4 \} 
  \cup \{ a_0,a_1,a_2,a_4 \}
  \cup \{ a_0,a_1,a_2,a_3 \} \qquad\\
 \cup \{ a_0,a_3,a_4,a_5 \}
  \cup \{ a_0,a_1,a_4,a_5 \}
  \cup \{ a_0,a_1,a_3,a_5 \},
\end{split}
\end{align}
the cone $I\!I$ corresponds to the triangulation
\begin{align}
 \{ a_1,a_2,a_3,a_4 \} \cup \{ a_1,a_3,a_4,a_5 \},
\end{align}
the cone $I\!I\!I$ corresponds to the triangulation
\begin{align}
 \{ a_1,a_2,a_3,a_5 \} \cup \{ a_2,a_3,a_4,a_5 \},
\end{align}
and the cone $IV$ corresponds to the triangulation
\begin{align}
\begin{split}
 \{ a_0,a_1,a_3,a_5 \} \cup \{ a_0,a_1,a_2,a_5 \}
  \cup \{ a_0,a_1,a_2,a_3 \} \qquad \\
  \cup \{ a_0,a_3,a_4,a_5 \} \cup \{ a_0,a_2,a_4,a_5 \}
  \cup \{ a_0,a_2,a_3,a_4 \}.
\end{split}
\end{align}

\subsection{The period domain}

By \cite[Theorem 3.1]{MR2962398},
the transcendental lattice $T_1$ of a very general member of the family
$\fY_{A_1} \to \XFone$
is isometric to $U \bot U(3)$.
The period domain for this family is described as follows:
Let
$
 R=M_2(\bZ)\cong \bZ^4
$
be a free abelian group of rank 4,
equipped with the symmetric bilinear form
\begin{align}
 \la v, w \ra=-\det(v+w)+\det(v)+\det(w).
\end{align}
Then $R$ is isometric to $U \bot U$, and
the corresponding domain is given by
\begin{align}
 \cD^+=\lc
  Z=\begin{pmatrix}
   \tau_1 \tau_2 & \tau_1 \\
   \tau_2 & 1
  \end{pmatrix}
  \relmid \tau_i \in \bH
 \rc \cong \bH \times \bH.
\end{align}
Indeed, we have
\begin{equation}
 \langle Z,Z \rangle=0, \quad
 \langle Z,\overline{Z} \rangle
  =-(\tau_1-\overline{\tau}_1) (\tau_2-\overline{\tau}_2) > 0.
\end{equation}
The group $\SL_2(\bZ)\times\SL_2(\bZ)$ acts
on $R$ by
\begin{equation}
 (A,B)\cdot v=A v B^{T}, \quad (A,B)\in \SL_2(\bZ)\times\SL_2(\bZ).
\end{equation}
This action induces an action on $\mathcal{D}$ defined by
\begin{equation}
 Z'=(A,B)\cdot Z, \quad
 AZB^{T} \sim Z', \quad
 (A,B)\in \SL_2(\bZ)\times\SL_2(\bZ),
\end{equation}
 where $Z\sim Z'$ if and only if $Z=\lambda Z'$
 for some $\lambda\in \bC^\times$.
For $A=\begin{pmatrix} a & b \\ c & d \end{pmatrix}\in \SL_2(\bZ)$,
 we have
\begin{equation}
 AZ
  =\begin{pmatrix}
    (a \tau_1 + b) \tau_2 & a \tau_1 + b \\
    (c \tau_1 + d) \tau_2 & c \tau_1 + d
   \end{pmatrix}
  \sim
   \begin{pmatrix}
    (A\cdot \tau_1) \tau_2 & A\cdot \tau_1 \\
    \tau_2 & 1
   \end{pmatrix}.
\end{equation}
Similarly, for
 $B \in\SL_2(\bZ)$,
we have
\begin{equation}
 ZB^{T}\sim
 \begin{pmatrix}
  \tau_1 (B\cdot \tau_2) & \tau_1 \\
  B\cdot \tau_2 & 1
 \end{pmatrix}. 
\end{equation}
Hence the identification $Z=(\tau_1,\tau_2)$
 is compatible with the actions of $\SL_2(\bZ)\times \SL_2(\bZ)$;
\begin{equation}
 (A,B)\cdot Z=(A\cdot \tau_1,B \cdot \tau_2).
\end{equation}
We define an involution $\rho$ of $R$ by $\rho(v)=v^T$
 and extend it linearly.
Then we have $(A,B)\rho=\rho(B,A)$ and the induced action of $\rho$
 on $\mathcal{D}$ is given by
 $\rho\cdot (\tau_1,\tau_2)=(\tau_2,\tau_1)$.

Now consider the sublattice $L$ of $R$ defined by
\begin{gather}
 L= \lc v=\begin{pmatrix} x & y \\ z & w \end{pmatrix} \in R \relmid
  \langle v,u \rangle \equiv 0 \bmod 3 \rc,
\end{gather}
where
$
 u=\SqMat{1}{0}{0}{0}.
$
Then we have $L \cong U \bot U(3)$.
The subgroup of $\SL_2(\bZ)\times \SL_2(\bZ)$
which preserves $L$ as a set
is given by
$
 \Gamma'=\Gamma_0(3)\times \Gamma_0(3).
$
One has
\begin{align}
 O(L) = \la \Gamma', \sigma \ra \rtimes \la \rho \ra,
\end{align}
where $\sigma \in O(L)$ is defined by
\begin{equation}
 \sigma=(S,S) \colon
 \SqMat{x}{y}{z}{w} \mapsto S \SqMat{x}{y}{z}{w} S^T
  =\SqMat{w/3}{-z}{-y}{3x}, \quad
 S=\frac{1}{\sqrt{3}} \SqMat{}{1}{-3}{}. 
\end{equation}
The discrete group $\Gamma^+$ is given by
\begin{align}
 \Gamma^+
  \cong \Ker \lb O^+(L) \to \Aut(L^\vee/L) \rb
  \cong \Gamma' \ltimes \la \rho \ra,
\end{align}
and the moduli space is given by
\begin{align}
 \cM
  = \cD^+/\Gamma^+
  = X_0'(3) \times X_0'(3) / C_2,
\end{align}
where $C_2 = \la \rho \ra = \bZ/2\bZ$
acts on the product of two copies
of the modular curve $X_0'(3) = \bH / \Gamma_0(3)$
by permutation.
Recall that the modular curve
$X'_0(3) =\bH/\Gamma_0(3)$
has two cusps and one orbifold point of order 3.
It can be compactified to $X_0(3) \cong \bP(1,3)$,
where the position of the cusps and the orbifold point
can be set to $0$, $1$, and $\infty$.
Hence the Baily-Borel-Satake compactification
$\cMbar$ of $\cM$ is obtained from $\bP(1,3,3)$
by the root construction along the image of the diagonal.
\pref{fg:x-plane1_1} shows the product $X_0(3) \times X_0(3)$,
and \pref{fg:q-plane1_1} shows the quotient
$\cMbar_1 = X_0(3) \times X_0(3) / \la \rho \ra$.
The dotted line in \pref{fg:x-plane1_1} is the diagonal,
which goes to the dotted line in \pref{fg:q-plane1_1},
where it has a generic stabilizer of order two.

\begin{figure}[t]
\begin{minipage}{.5 \linewidth}
\centering
\input{x-plane1_1.pst}
\caption{$X_0(3) \times X_0(3)$}
\label{fg:x-plane1_1}
\end{minipage}
\begin{minipage}{.5 \linewidth}
\centering
\input{q-plane1_1.pst}
\caption{$\cMbar_1$}
\label{fg:q-plane1_1}
\end{minipage}
\end{figure}

\subsection{The period map}

The space of Laurent polynomials is given by
\begin{align}
 \bC^{A_1} = \lc a_0 + a_1 x + a_2 y + a_3 z
  + \frac{a_4}{x z} + \frac{a_5}{y z} \rc,
\end{align}
and the discriminant is given by
\begin{align*}
 \bsDelta_1 &= 
a_0^6+54 a_1 a_3 a_4 a_0^3+54
   a_2 a_3 a_5 a_0^3+729 a_1^2
   a_3^2 a_4^2+729 a_2^2 a_3^2
   a_5^2-1458 a_1 a_2 a_3^2
   a_4 a_5.
\end{align*}
The dense torus $\bL^\vee_{\bCx}$
of the secondary stack $\XFone$ can be written as
$\Spec \bC[\lambda^{\pm 1}, \mu^{\pm 1}]$ where
\begin{align}
 \lambda &= \frac{a_1 a_3 a_4}{a_0^3}, \\
 \mu &= \frac{a_2 a_3 a_5}{a_0^3}.
\end{align}
\begin{comment}
\begin{shaded}
The Horn-Kapranov uniformization
\begin{align}
 h : \bP^1 \to \bL \otimes \bCx, \ 
  [\lambda_1:\lambda_2] \mapsto (\lambda, \mu)
\end{align}
of the discriminant $\bsnabla_1$ is given by
\begin{align}
 \lambda &= -\frac{1}{27} \frac{\lambda_1^3}{(\lambda_1+\lambda_2)^2}, \\
 \mu &= -\frac{1}{27} \frac{\lambda_2^3}{(\lambda_1+\lambda_2)^2}.
\end{align}
\begin{align}
 \frac{1}{27}\lambda(\lambda_1,\lambda_2) &=
  (-3\lambda_1-3\lambda_2)^{-3}
  \lambda_1^1 \lambda_2^0
  (\lambda_1+\lambda_2)^1
  \lambda_1^1 \lambda_2^0
   = - \frac{1}{27} \frac{\lambda_1^3}{(\lambda_1+\lambda_2)^2} \\
 \frac{1}{27}\mu(\lambda_1,\lambda_2) &=
  (-3\lambda_1-3\lambda_2)^{-3}
  \lambda_1^0 \lambda_2^1
  (\lambda_1+\lambda_2)^1
  \lambda_1^0 \lambda_2^1
   = - \frac{1}{27} \frac{\lambda_2^3}{(\lambda_1+\lambda_2)^2}
\end{align}
which clearly satisfies
\begin{align}
 1+2(\lambda+\mu)+(\lambda-\mu)^2=0.
\end{align}
\end{shaded}
\end{comment}

The period which is holomorphic around
$\lambda=\mu=0$
is given by
\begin{align}
 \eta_1(\lambda,\mu)
  &= \sum_{n,m=0}^\infty (-1)^{m+n}
   \frac{(3n+3m)!}{(n!)^2 (m!)^2 (n+m)!} \lambda^n \mu^m.
\end{align}
Recall the classical relation
(cf.~e.g.~\cite[(8)]{MR2514458})
\begin{align} \label{eq:Gausstensor}
 F_4 (a,b,c,a+b-c+1; x(1-y), y(1-x))={}_2F_1(a,b,c;x)  {}_2F_1(a,b,a+b-c+1;x) 
\end{align}
between Appell's function
\begin{align}
 F_4(a,b,c_1,c_2;z,w)
  = \sum_{m,n=0}^\infty
   \frac{(a)_{m+n} (b)_{m+n}}{(c_1)_n (c_2)_m n! m! } z^n w^m
\end{align}
and Gauss hypergeometric function
\begin{align}
 {}_2F_1(a,b,c;z)=\sum_{n=0}^\infty \frac{(a)_n (b)_n}{ (c)_n n!} z^n,
\end{align}
where
$
 (\alpha)_n=\alpha (\alpha+1) \cdots (\alpha+n-1)
$
is the Pochhammer symbol.
An elementary manipulation shows
\begin{align}
 \eta_1(\lambda, \mu)
  &= F_4\Big(\frac{1}{3},\frac{2}{3},1,1;-27 \lambda, -27 \mu\Big),
\end{align}
which is equal to
\begin{align}
 {}_2F_1\Big(\frac{1}{3},\frac{2}{3},1;x\Big)
   {}_2F_1\Big(\frac{1}{3},\frac{2}{3},1;x\Big) 
\end{align}
by \eqref{eq:Gausstensor} where
$
 -27 \lambda=x(1-y)
$
and
$
 -27 \mu=y(1-x).
$
Gauss hypergeometric function
\begin{align}
 f_1(x)=\displaystyle {}_2F_1\Big(\frac{1}{3},\frac{2}{3},1;x\Big)
\end{align}
is a solution to Gauss hypergeometric differential equation
\begin{align} \label{Gauss(1/3,2/3,1)}
 {}_2E_1(\frac{1}{3},\frac{2}{3},1):x(1-x) \frac{d^2 u}{d x^2}
  +(1-2x)\frac{du}{dx} -\frac{2}{9} u=0.
\end{align}
This differential equation has regular singularity
at $x=0, 1, \infty$.
By choosing a suitable solution $f_2(x)$
to \eqref{Gauss(1/3,2/3,1)}
which is holomorphic in a neighborhood
of $[0+\varepsilon,1-\varepsilon] \subset \bR$ in $\bC$
for sufficiently small $\varepsilon$,
one obtains a map
\begin{align}
 x\mapsto \frac{f_2(x)}{f_1(x)}=s (x) \in \bH
\end{align}
which is defined on $(0,1) \subset \bR$ and satisfies
$
 s(0)=\sqrt{-1} \infty,
$
$
 s(1)= 0,
$
$
 s(\infty)=\displaystyle \frac{1}{2}+\frac{\sqrt{3}}{6} \sqrt{-1}.
$
This map can be extended
to a multivalued function
from $\bP^1 \setminus \{0, 1, \infty\}$ to $\bH$.
%
%
This multi-valued map $s$ sends $\bH \subset \bP^1$
to a hyperbolic triangle with angles $0$, $0$, and $\dfrac{1}{3} \pi$
in $\bH$.
The monodromy of $s$
along the closed paths $\gamma_0$, $\gamma_1$ and $\gamma_\infty$
going around $x = 0$, $1$ and $\infty$
are given by
\begin{align}\label{eq:taka}
\begin{cases}
 (\gamma_0)_* (s) =s+1, \\[2mm]
 (\gamma_1)_* (s) = \dfrac{s}{-3s +1}, \\[3mm]
 (\gamma_\infty)_* (s) = \dfrac{s-1}{3s -2}.
\end{cases}
\end{align}
The inverse map
$
 \bH \to \bP^1
$
sending $s$ to $x$
is a modular function with respect to $\Gamma_0(3)$.

The period for the family $\fY_1$ has the form
\begin{align}
 (\lambda,\mu) \mapsto
  \eta= [\eta_{11} : \eta_{12}: \eta_{13}: \eta_{14}]
  \in \cD_1=\lc \eta \in \bP(T_1) \relmid (\eta, \eta) = 0, \ 
   (\eta, \overline{\eta}) > 0  \rc
\end{align}
where
$
 T_1=U \oplus U(3).
$
The connected component $\cD_1^+ \subset \cD_1$
can be identified with $\bH \times \bH$ by
\begin{align}
 (z_1,z_2)
  \mapsto [\eta_1:\eta_2 : \eta_ 3: \eta_4]
  = [3z_1z_2 : -1 : z_1 : z_2].
\end{align}
The Gauss-Manin system for $\cY_1$ is
Appell's hypergeometric differential equation of rank 4,
and the periods are given by
\begin{align*}
\eta_{11}= 3 f_2(x) f_2(y),
\quad
\eta_{12}=-  f_1(x) f_1(y),
\quad
\eta_{13}= f_2(x) f_1(y),
\quad
\eta_{14}=f_1(x) f_2(y)
\end{align*}
This period map extends
to the $xy$-plane,
which is a double cover of the $\lambda \mu$ plane;
\begin{align}
 (x,y)
  \mapsto [\eta_{11}:\eta_{12}:\eta_{13}:\eta_{14}]
  = [3s(x)s(y):-1 : s(x): s(y)].
\end{align}
The inverse map
\begin{align}
 \bH \times \bH
  \to \bP^1 \times \bP^1, \quad
 (s_1,s_2) \mapsto (x(s_1), y(s_2))
\end{align}
is given by
\begin{align*}
\begin{cases}
 \lambda (s_1,s_2) =\dfrac{x(s_1)(y(s_2)-1)}{27}, \\[3mm]
 \mu (s_1,s_2)= \dfrac{y(s_1)(x(s_2)-1)}{27}.
\end{cases}
\end{align*}

\subsection{Mirror symmetry and monodromy}

The polar dual polytope is given by
\begin{align*}
 \Deltav_1 &= \Conv \lc
  (2,2,-1),(2,-1,-1),(-1,-1,-1),(-1,2,-1),(-1,-1,2)
 \rc.
\end{align*}
The ambient space $\Xv$ for the mirror family $\cYv$
is the toric weak Fano 3-fold of Picard number 2,
which is obtained as a crepant resolution
of a toric Fano 3-fold of Picard number 1
with an ordinary double point.

The Picard lattice of a very general member of $\cYv$
is generated by the restrictions $E_1$ and $E_2$
of the toric divisors $D_4$ and $D_5$
of the ambient space,
and one has
\begin{align}
 M_{\Delta_1} &= \Pic Y_1
  = E_8 \bot E_8 \bot U(3), \\
 M_{\Delta_1}^\bot &= U \bot U(3), \\
 M_{\Deltav_1}
  &= \Pic \Yv_1
  = U(3), \\
 M_{\Deltav_1}^\bot
  &= U \bot E_8 \bot E_8 \bot U(3),
\end{align}
so that \pref{cj:Dolgachev} holds in this case.

The numerical Grothendieck group $\cN(\Yv)$
is isometric to $U \bot \Pic(\Yv)$, and
the nef cone of $\Xv$ is generated by $D_4+D_5$ and $D_5$.
One can show that the monodromies
around $q_1=\lambda$ and $q_2=\lambda^{-1} \mu$
are given by
$\varphi_{e,e_1+e_2}$ and $\varphi_{e,e_2}$
just as in the case of $A_0$.
For the other crepant resolution,
the nef cone is generated by $D_4$ and $D_4 + D_5$,
and the monodromies are given by
$\varphi_{e,e_1}$ and $\varphi_{e,e_1+e_2}$.
As a result,
the toroidal compactification is given by the fan in $N_\bR$
whose one-dimensional cones are spanned by
$e_1$, $e_1+e_2$, and $e_2$.
This blows up the intersection point
of two components of the cusp
as shown in \pref{fg:q-plane1_2},
and the resulting stack is precisely
the stack $\XtildeFone$ obtained from $\XFone$
by the root construction along the strict transform
of the diagonal in $X_0(3) \times X_0(3)$.

\begin{figure}[t]
\begin{minipage}{\linewidth}
\centering
\input{q-plane1_2.pst}
\caption{$\XFone$}
\label{fg:q-plane1_2}
\end{minipage}
\end{figure}

\bibliographystyle{amsalpha}
\bibliography{bibs.bib}

\noindent
Kenji Hashimoto

School of Mathematics,
Korea Institute for Advanced Study,
85 Hoegiro,
Dongdaemun-gu,
Seoul,
130-722,
Korea

{\em e-mail address}\ : \ hashimoto@kias.re.kr

\ \vspace{-5mm} \\

%
%
%

\noindent
Atsuhira Nagano

c.o. Professor Kimio Ueno,
Department of Mathematics,
Waseda University,
Okubo 3-4-1,
Shinjuku-ku
Tokyo,
169-8555,
Japan

{\em e-mail address}\ : \ ornithology@akane.waseda.jp

\ \vspace{-5mm} \\

\noindent
Kazushi Ueda

Department of Mathematics,
Graduate School of Science,
Osaka University,
Machikaneyama 1-1,
Toyonaka,
Osaka,
560-0043,
Japan.

{\em e-mail address}\ : \  kazushi@math.sci.osaka-u.ac.jp

\end{document}